\documentclass[psamsfonts,12pt]{amsart}

\usepackage{layout}
\usepackage[makeroom]{cancel}
\usepackage{bbm}

\usepackage{amsfonts,dsfont}
\usepackage{amsmath}
\usepackage{array}
\usepackage{amssymb,bm}
\usepackage{amsthm}
\usepackage{graphicx} 
\usepackage{mathtools}
\usepackage{multicol}
\usepackage{mathrsfs}
\usepackage{enumerate,enumitem}
\usepackage[rgb,dvipsnames]{xcolor}
\usepackage{float}
\usepackage[normalem]{ulem}
\usepackage{circuitikz}
\usepackage{cite}
\usepackage[colorlinks,linkcolor=blue,citecolor=blue,pagebackref,hypertexnames=false, breaklinks]{hyperref}

\usepackage{float}

\usepackage{fancyhdr}
\allowdisplaybreaks[4]
\newtheorem{theorem}{Theorem}[section]

\newtheorem{definition}[theorem]{Definition}
\newtheorem{proposition}[theorem]{Proposition}

\newtheorem{lemma}[theorem]{Lemma}
\newtheorem{remark}[theorem]{Remark}
\newtheorem{example}[theorem]{Example}

\newtheorem{examples}[theorem]{Examples}

\newtheorem{open question}[theorem]{Open Question}
\newtheorem{c/p}[theorem]{Conjecture/Proposition}

\def\vint{\mathop{\mathchoice%
 {\setbox0\hbox{$\displaystyle\intop$}\kern 0.22\wd0%
 \vcenter{\hrule width 0.6\wd0}\kern -0.82\wd0}%
 {\setbox0\hbox{$\textstyle\intop$}\kern 0.2\wd0%
 \vcenter{\hrule width 0.6\wd0}\kern -0.8\wd0}%
 {\setbox0\hbox{$\scriptstyle\intop$}\kern 0.2\wd0%
 \vcenter{\hrule width 0.6\wd0}\kern -0.8\wd0}%
 {\setbox0\hbox{$\scriptscriptstyle\intop$}\kern 0.2\wd0%
 \vcenter{\hrule width 0.6\wd0}\kern -0.8\wd0}}%
 \mathopen{}\int}

\usepackage{geometry}
\title{Fractional Gaussian fields  on the Sierpinski gasket and related fractals}

\author{Fabrice Baudoin}
\address{F.B: Department of Mathematics, University of Connecticut, Storrs, CT 06269}
\author{C\'eline Lacaux}
\address{C.L.: Avignon Universit\'e, LMA EA 2151, 84000 Avignon, France}
\thanks{F. Baudoin is supported in part by NSF Grant DMS-1901315}

\begin{document}

\maketitle

\begin{abstract}
We define and study a fractional Gaussian field $X$ with Hurst parameter  $H$ on the Sierpi\'nski gasket $K$ equipped with its Hausdorff measure $\mu$.  It appears as a  solution, in a weak sense,  of the equation $(-\Delta)^s X =W$ where $W$ is a Gaussian white noise on $L_0^2(K,\mu)$, $\Delta$ the Laplacian on $K$ and $s= \frac{d_h+2H}{2d_w}$, where $d_h$ is the Hausdorff dimension of $K$ and $d_w$ its walk dimension. The construction of those fields is then extended to other fractals including the Sierpi\'nski carpet. 
\end{abstract}

\tableofcontents

\section{Introduction}

For $s \ge 0$, consider in $\mathbb{R}^n$ the Gaussian random measure
\begin{equation}\label{intro X}
X=(-\Delta)^{-s} W,
\end{equation}
where $W$ is a white noise (i.e. a Gaussian random measure with intensity the Lebesgue measure) and $\Delta$ the Laplace operator on $\mathbb{R}^n$. The expression \eqref{intro X} has of course to be understood in a distributional sense (see \cite{MR3466837} for the details) and means that for every $f$ in the Schwartz space $\mathcal{S}(\mathbb{R}^n)$ of smooth and rapidly decreasing functions one has
\[
\int_{\mathbb{R}^n} (-\Delta )^s f (x) X (dx)=\int_{\mathbb{R}^n}  f (x) W (dx).
\]

This class of Gaussian measures includes the following popular examples which are thoroughly presented in the survey paper \cite{MR3466837}: white noise ($s=0$), Gaussian free field ($s=1/2$),  log-correlated Gaussian field ($s=\frac{n}{4}$) and fractional random measures ($\frac{n}{4}< s<\frac{n}{4}+\frac{1}{2}$). Actually, in the range $\frac{n}{4}< s<\frac{n}{4}+\frac{1}{2}$,  the Gaussian random measure $X$ admits a density with respect to the Lebesgue measure which is the fractional Brownian motion indexed by $\mathbb{R}^n$. The Hurst parameter $H$ of this fractional Brownian motion  is given by $H=2 s -\frac{n}{2}$.
In the present paper, we are interested in generalizing those fields on fractals in the range corresponding to the fractional Brownian motions. For simplicity of the presentation we carry out explicitly  and in details the analysis in the case of the Sierpi\'nski gasket but as we shall discuss in the last section of the paper, our analysis extends to more general fractals. The main result is the following:

\begin{theorem}\label{theo intro}
Let $K$ be the Sierpi\'nski gasket with normalized self-similar Hausdorff measure $\mu$ and Laplacian $\Delta$. Denote $d_h$ the Hausdorff dimension of $K$ and $d_w$ its walk dimension.  Let $W$ be a white noise on $L^2_0(K,\mu)$. Then,  if $\frac{d_h}{2d_w} < s < 1-\frac{d_h}{2d_w}$, there exists a Gaussian random field $(X(x))_{x \in K}$   which is  H\"older continuous  with exponent $H^{-}$ where
\[
H= sd_w -\frac{d_h}{2},
\]
and such that for every $f$ which is in the $L^2_0$ domain of the operator $(-\Delta )^s  $
\[
\int_{K} (-\Delta )^s f (x) X(x)  d\mu(x)=\int_{K}  f (x) W (dx).
\]
\end{theorem}

By $\gamma^{-}$-H\"older we mean  $(\gamma-\varepsilon)$-H\"older continuous for $\varepsilon>0$. In the range $\frac{d_h}{2d_w} < s < 1-\frac{d_h}{2d_w}$ it is therefore natural to call $X$ a fractional Brownian motion indexed by $K$ and with Hurst parameter $H= sd_W -\frac{d_h}{2}$. For the Sierpi\'nski gasket the explicit values $d_h=\frac{\ln 3}{\ln 2}$ and  $d_w=\frac{\ln 5}{\ln 2}$ are known. The borderline case $s =\frac{d_h}{2d_w}$ would correspond to the case of a log-correlated field on the gasket. Such field can not be defined pointwise but only in a distributional sense. We let its study for possible later research.

Since their introduction in \cite{Kolmogorov} and \cite{mandelbrot}, fractional Brownian motions and fields have attracted a lot of interest, both from theoretical or more applied viewpoints, see \cite{MR3466837, taqqu,cohenistas2013} and the references therein. Following the definition by P. L\'evy \cite{levy} of the Brownian field on the sphere,  J.  Istas defined in \cite{Istas1, istas2} the fractional Brownian field on  manifolds or more generally metric spaces, as a Gaussian field whose covariance is given by
\begin{align}\label{cov intro}
\frac{1}{2} ( d(x,o)^{2H}+d(y,o)^{2H}-d(x,y)^{2H}),
\end{align}
where $o$ is a fixed point in the space and $d$ the distance. Applying this definition for the Euclidean distance on the Sierpi\'nski gasket, which is a compact set isometrically embedded in the plane, is not interesting since it simply yields a field which is the restriction to the gasket of the usual fractional Brownian field on the plane. It would be more natural to use for the distance $d$ the so-called harmonic shortest path metric which is for instance defined by J. Kigami in \cite{kigami4}. For this choice of the distance, it is not clear to the authors what is the exact range of the parameters $H$ for which the function \eqref{cov intro} is indeed a covariance function. Our construction of the fractional Brownian field, which is instead based on the study of fractional Riesz kernels is similar to the construction of fractional fields on manifolds by Z. Gelbaum  \cite{Gelbaum} and adopt the viewpoint about fractional fields which is given in \cite{MR3466837}. One advantage of working with fractional fields constructed from the white noise using fractional Riesz kernels is the availability of all the harmonic analysis tools that can be developed on the underlying space. In the case of fractals like the  Sierpi\'nski gasket such tools have extensively been developed in the last few decades; we mention for instance the references \cite{Kigami3, Shima, Strichartz1, Strichartz2} and the book \cite{kigami}.\\

Our paper is organized as follows.  In Section 2, we study on the Sierpi\'nski gasket the properties of the operator  $(-\Delta)^{-s}$ where $\Delta$ is the Laplacian on the gasket, as defined in \cite{Kigami3}. One of the main results of the section is Theorem \ref{Holder Riesz} that implies   that for $s \in \left( \frac{d_h}{2d_w} , 1- \frac{d_h}{2d_w}\right)$ and $f \in L^2(K,\mu)$ one  can pointwisely define $(-\Delta)^{-s} f$ and that one has for every $x,y \in K$,
\[
| (-\Delta)^{-s} f (x) - (-\Delta)^{-s} f (y) | \le  Cd(x,y)^{sd_w-\frac{d_h}{2}} \| f \|_{L^2(K,\mu)}.
\]
In particular, in the range  $s \in \left( \frac{d_h}{2d_w} , 1- \frac{d_h}{2d_w}\right)$, the operator $(-\Delta)^{-s}$ maps  $L^2(K,\mu)$ into the space of bounded and $\left(sd_w-\frac{d_h}{2}\right)$-H\"older continuous functions. This regularization property allows us to define and study in Section 3, the fractional Brownian field as $X:=(-\Delta)^{-s} W$ where $W$ is a white noise and then to prove Theorem~\ref{theo intro}. Tools to the study of the regularity of Gausian fields are widely available in the literature, see for instance \cite{MR346884, MR1462329,kono1970,MR2531090} and references therein. In our situation, a key step is Theorem  \ref{regularite}  where we prove, using a  Garsia-Rodemich-Rumsey inequality for fractals (see lemma 6.1 of \cite{BP88}), that there exists a modification $X^*$ of $X$ such that  
	$$
	\lim_{\delta \to 0}\underset{\underset{x,y\in K}{\scriptsize d(x,y)}\le \delta}{\sup}\  \frac{\left| X^*(x)-X^*(y)\right|}{d(x,y)^H \sqrt{\left|\ln d(x,y)\right|}} <+\infty	$$ 

with $H= sd_w -\frac{d_h}{2}$. Finally, at the end of the section, we prove that the fractional field we constructed is invariant by the symmetries of the gasket and moreover satisfies a natural scaling property related to the self-similar structure of the gasket. In the final Section 4, we extend our results to the context of fractional spaces, which are a class of Dirichlet spaces introduced by Barlow in  \cite{Bar98}.

\section{Fractional Riesz kernels on the Sierpi\'nski gasket}

\subsection{Definition of the gasket}

We first recall the definition of the Sierpi\'nski gasket. For further details we refer to the book by Kigami \cite{kigami}.
In $\mathbb{R}^2 \simeq \mathbb C$, consider the triangle with vertices $q_0=0$, $q_1=1$ and $q_2=e^{\frac{i\pi}{3}}$. For $i=1,2,3$, consider the map
\[
F_i(z)=\frac{1}{2}(z-q_i)+q_i.
\]

\begin{definition}
The Sierpi\'nski gasket is the unique non-empty compact set $K \subset \mathbb C$ such that 
\[
K=\bigcup_{i=1}^3 F_i (K).
\]
\end{definition}

The Hausdorff dimension of $K$ with respect to the Euclidean metric (denoted $d(x,y)=| x - y |$ in this paper) is given by $d_h=\frac{\ln 3}{\ln 2}$. A (normalized) Hausdorff measure on $K$ is given by the Borel measure $\mu$ on $K$ such that for every $i_1, \cdots, i_n \in \{ 1,2,3 \} $,
\[
\mu \left(  F_{i_1} \circ \cdots \circ F_{i_n}  (K)\right)=3^{-n}.
\]

\begin{figure}[htb]\centering
 	\includegraphics[trim={60 10 180 60},height=0.25\textwidth]{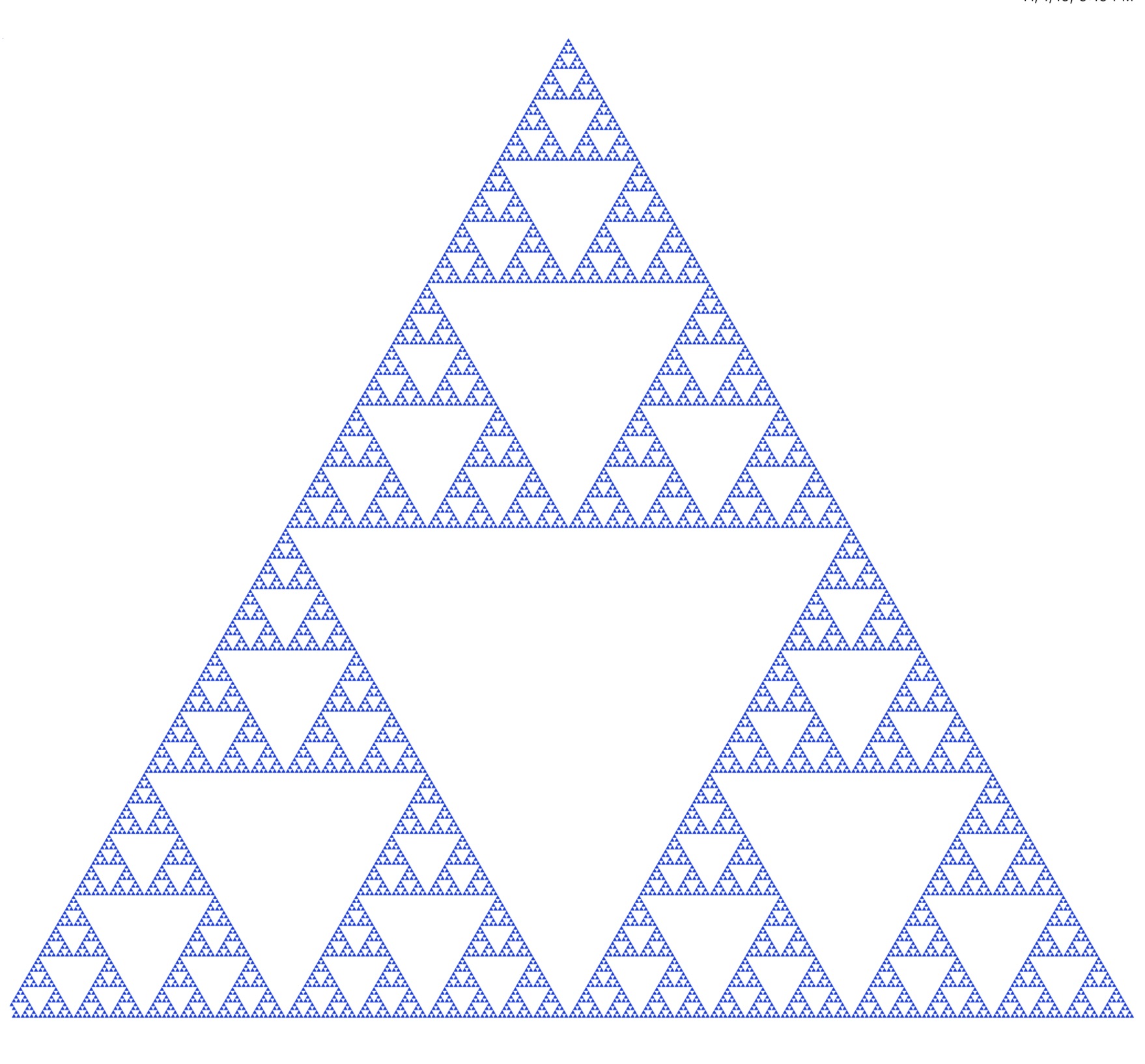}
 	\caption{Sierpi\'nski gasket.} \label{fig-SG}
 \end{figure}

This measure $\mu$ is $d_h$-Ahlfors regular, i.e. there exist constants $c,C>0$ such that for every $x \in K$ and $r \in [0, \mathrm{diam} (K) ]$,
\begin{equation}
\label{Ahlfors}
c r^{d_h} \le \mu (B(x,r)) \le C r^{d_h},
\end{equation}

where we denote by $\mathrm{diam} (K)$ the diameter of $K$ and by $B(x,r)$ the metric ball with center $x$ and radius~$r$.

\subsection{Canonical Dirichlet form and heat kernel estimates}
One can construct a canonical Dirichlet form and associated Laplacian $\Delta$ on the Sierpi\'nski gasket by using a graph approximation of the gasket. Denote $V_0= \{ q_0,q_1,q_2 \}$, $V_n=\cup_{i_1,\cdots, i_n} F_{i_1} \circ \cdots \circ F_{i_n} (V_0)  $ and
\[
V_* =\bigcup_{ n \ge 0} V_n
\]

For $f \in \mathbb{R}^{V_n}$, one can consider the quadratic form
\[
\mathcal{E}_n (f,f)=\frac{1}{2} \left( \frac{5}{3} \right)^n \sum_{i_1,\cdots, i_n} \sum_{x,y \in V_0} \left( f(F_{i_1} \circ \cdots \circ F_{i_n} (x)) - f(F_{i_1} \circ \cdots \circ F_{i_n} (y)) \right)^2
\]

Define then
\[
\mathcal{F}_* = \left\{ f \in \mathbb{R}^{V_*}, \lim_{n \to \infty}   \mathcal{E}_n (f,f) <+\infty \right\} 
\]
and for $f \in \mathcal{F}$,
\begin{align}\label{dirichlet limit}
\mathcal{E} (f,f) =\lim_{n \to \infty}   \mathcal{E}_n (f,f).
\end{align}

It is possible to prove that any function $f \in \mathcal{F}_*$ can uniquely be extended into a continuous function defined on the whole $K$. We denote by $\mathcal F$ the set of such extensions. One has then the following theorem, see the book by Kigami \cite{kigami}.

\begin{theorem}
$(\mathcal{E},\mathcal{F})$ is a local regular Dirichlet form on $L^2(K,\mu)$ with the following property: for every $f,g \in \mathcal F$
\[
\mathcal{E} (f,g)=\frac{5}{3} \sum_{i=1,2,3} \mathcal{E} (f \circ F_i , g \circ  F_i ).
\]
\end{theorem}

The semigroup $\{P_t\}$ associated with $\mathcal{E}$ is stochastically complete (i.e. $P_t 1=1$) and, from \cite{BP88},   has a jointly continuous heat kernel $p_t(x,y)$ with respect to the reference measure $\mu$ satisfying, for some $c_{1},c_{2}, c_3, c_4 \in(0,\infty)$,
 \begin{equation}\label{eq:subGauss-upper}
 c_{1}t^{-d_{h}/d_{w}}\exp\biggl(-c_{2}\Bigl(\frac{d(x,y)^{d_{w}}}{t}\Bigr)^{\frac{1}{d_{w}-1}}\biggr) 
 \le p_{t}(x,y)\leq c_{3}t^{-d_{h}/d_{w}}\exp\biggl(-c_{4}\Bigl(\frac{d(x,y)^{d_{w}}}{t}\Bigr)^{\frac{1}{d_{w}-1}}\biggr)
 \end{equation}
 for every \ $(x,y)\in K \times K$ and $t\in\bigl(0,1)$. 
 
 The exact values of $c_1,c_2,c_3,c_4$ are irrelevant in our analysis.   As above, the parameter $d_h=\frac{\ln 3}{\ln 2}$ is the Hausdorff dimension. The parameter $d_w=\frac{\ln 5}{\ln 2}$  is called the walk dimension. Since $d_w > 2$, one speaks of sub-Gaussian  heat kernel  estimates. 

\subsection{Fractional Riesz kernels}

Let $\Delta$ denotes the generator of the Dirichlet form $\mathcal E$, i.e. $\Delta$ is the Laplacian on $K$.
Our goal in this section is to study the operators $(-\Delta)^{-s}$, $s >0$, defined on $L^2_0(K,\mu)$ where 
\[
L^2_0(K,\mu)=\left\{ f \in L^2(K,\mu), \int_K f d\mu=0 \right\}.
\]

From \cite{kigami}, the heat kernel $p_t(x,y)$ admits a uniformly convergent spectral expansion:
\begin{align}\label{spectral}
p_t(x,y)=1+\sum_{j=1}^{+\infty} e^{-\lambda_j t} \Phi_j(x) \Phi_j(y)
\end{align}
where $0<\lambda_1\le \lambda_2\le  \cdots \le \lambda_n \le \cdots$ are the eigenvalues of $-\Delta$ and the $\Phi_j \in \mathcal F$, $j\ge 1$, an orthonormal basis of $L^2_0(K,\mu)$ such that  
\[
\Delta \Phi_j =-\lambda_j \Phi_j.
\]
Notice that $\Phi_j \in \mathcal{F}$ and thus is continuous.

It is known from \cite{Shima} that the counting function of the eigenvalues:
\[
N(t)=\mathbf{Card} \{ \lambda_j \le t \}
\]
satisfies
\[
N(t) \sim \Theta(t) t^{d_h/d_w}
\]
when $t \to +\infty$ where $\Theta$ is a function bounded away from 0. In particular,
\[
\sum_{j=1}^{+\infty} \frac{1}{\lambda_j^{2s}} <+\infty
\]
whenever $s>\frac{d_h}{2d_w}$. For $s>\frac{d_h}{2d_w}$, the operator $(-\Delta)^{-s} $ is then defined as the bounded operator $
(-\Delta)^{-s} : L^2_0(K,\mu) \to L^2_0(K,\mu)$ given by
\[
(-\Delta)^{-s}f =\sum_{j=1}^{+\infty} \frac{1}{\lambda_j^s} \left( \int_K  \Phi_j(y) f(y) d\mu(y)\right) \Phi_j.
\]
From this definition, the function $(-\Delta)^{-s}f$ is thus a priori only defined $\mu$ a.e. We will prove in this section and the next one that it actually admits a H\"older continuous version, see Remark \ref{continuous version} and Theorem \ref{Holder Riesz}. To this end, we first collect basic heat kernel estimates.

\begin{lemma}\label{cont:pt:tgd}
There exists a constant $C>0$ such that for every $x,y \in K$ and $t \ge 1$,
\[
| p_t(x,y) -1 | \le C e^{-\lambda_1 t},
\]
where $\lambda_1 >0$ is the first non-zero eigenvalue of $K$.
\end{lemma}

\begin{proof}
As already noted, the heat kernel $p_t(x,y)$ admits a uniformly convergent spectral expansion:
\begin{align}\label{spectral2}
p_t(x,y)=1+\sum_{j=1}^{+\infty} e^{-\lambda_j t} \Phi_j(x) \Phi_j(y).
\end{align}
Since the $\Phi_j$'s are eigenfunctions, one has for any $t > 0$,
\[
\Phi_j(x)=e^{\lambda_j t}  \int_K p_t (x,y) \Phi_j(y) d\mu(y).
\]
Thus, from Cauchy-Schwarz inequality, we have for every $t >0$
\begin{align*}
| \Phi_j(x) | & \le e^{\lambda_j t} \left(  \int_K p_t (x,y)^2  d\mu(y) \right)^{1/2}\left(  \int_K  \Phi_j(y)^2 d\mu(y)\right)^{1/2}=e^{\lambda_j t}  p_{2t}(x,x)^{1/2}.
\end{align*}
In particular, choosing $t=1/4$ and using \eqref{eq:subGauss-upper}, one obtains that there exists a constant $C>0$ such that for every $x \in K$,
\[
| \Phi_j(x) | \le C e^{\lambda_j /4}.
\]
Coming back to the expansion \eqref{spectral2} one obtains that for every $x,y \in K$ and $t \ge 1$,
\[
| p_t(x,y) -1 | \le \sum_{j=1}^{+\infty} e^{-\lambda_j t} e^{\lambda_j /2}\le {e}^{-\lambda_1 t}{e}^{\lambda_1}\sum_{j=1}^{+\infty} {e}^{-\lambda_j/2},
\]
which concludes the proof. 
\end{proof}

\begin{lemma}\label{lemma 1}
For any $s>0$ and $x,y \in K$, $ x\neq y$, the integral
\[
 \int_0^{+\infty} t^{s-1} (p_t(x,y) -1)dt
\]
is absolutely convergent. Moreover, if $s > \frac{d_h}{d_w}$, the integral is also convergent for $x=y$.
\end{lemma}

\begin{proof}Thanks to the heat kernel upper bound  \eqref{eq:subGauss-upper}, 
the integral $\int_0^{1} t^{s-1} |p_t(x,y) -1|dt$ is finite for any $s>0$ when $x \neq y$  and for $s > \frac{d_h}{d_w}$ when $x=y$. Moreover, for any $x,y\in K$, the integral $\int_1^{+\infty} t^{s-1} |p_t(x,y) -1|dt$ is also finite thanks to lemma \ref{cont:pt:tgd}.
\end{proof}

We are now ready for the definition of the fractional Riesz kernels:

\begin{definition}
For a parameter $s>0$, we define the fractional Riesz kernel $G_s $ by 
\begin{align}\label{green function}
G_s(x,y)=\frac{1}{\Gamma(s)} \int_0^{+\infty} t^{s-1} (p_t(x,y) -1)dt, \quad x,y \in K, \, x\neq y.
\end{align}
with $\Gamma$ the gamma function.
\end{definition}

We will be interested in the integrability properties of $G_s$. The following estimates are therefore important.

\begin{proposition}\label{estimate G}
\
\begin{enumerate}
\item If  $s \in (0, d_h/d_w)$, there exists a constant $C >0$ such that for every  $x,y \in K$, $x \neq y$,
\[
| G_s(x,y) | \le \frac{C}{d(x,y)^{d_h-sd_w}}.
\]
\item If $s = d_h/d_w$, there exists a constant $C >0$ such that for every  $x,y \in K$, $x \neq y$
\[
| G_s(x,y) | \le C | \ln d(x,y) |.
\]
\item If  $s > d_h/d_w$, there exists a constant $C >0$ such that for every  $x,y \in K$, 
\[
 | G_s(x,y)| \le C.
\]
\end{enumerate}

\end{proposition}

\begin{proof}
We have
\begin{align*}
G_s(x,y)& =\frac{1}{\Gamma(s)} \int_0^{+\infty} t^{s-1} (p_t(x,y) -1)dt \\
 & = \frac{1}{\Gamma(s)} \int_0^{1} t^{s-1} (p_t(x,y) -1)dt+ \frac{1}{\Gamma(s)} \int_1^{+\infty} t^{s-1} (p_t(x,y) -1)dt.
\end{align*}
The integral $\int_1^{+\infty} t^{s-1} (p_t(x,y) -1)dt$ can uniformly be bounded on $K \times K$  by a constant using lemma~\ref{cont:pt:tgd}, so we just need to uniformly estimate the integral $\int_0^{1} t^{s-1} p_t(x,y)dt$. Thanks to the heat kernel upper bound  \eqref{eq:subGauss-upper} we have:
\begin{align*}
\int_0^{1} t^{s-1} p_t(x,y)dt& \le  c_{3} \int_0^1 t^{s-1-d_{h}/d_{w}}\exp\biggl(-c_{4}\Bigl(\frac{d(x,y)^{d_{w}}}{t}\Bigr)^{\frac{1}{d_{w}-1}}\biggr) dt .
\end{align*}
We now divide our analysis depending on the value of $s$. If $s >d_{h}/d_{w}$, one can simply bound
\[
\int_0^1 t^{s-1-d_{h}/d_{w}}\exp\biggl(-c_{4}\Bigl(\frac{d(x,y)^{d_{w}}}{t}\Bigr)^{\frac{1}{d_{w}-1}}\biggr) dt \le \int_0^1 t^{s-1-d_{h}/d_{w}} dt.
\]
If $s < d_{h}/d_{w}$, using the change of variable $t= u d(x,y)^{d_{w}}$, one sees that
\begin{align*}
& \int_0^{1} t^{s-1-d_{h}/d_{w}}\exp\biggl(-c_{4}\Bigl(\frac{d(x,y)^{d_{w}}}{t}\Bigr)^{\frac{1}{d_{w}-1}}\biggr) dt \\
 =&d(x,y)^{sd_w-d_h} \int_0^{1/d(x,y)^{d_{w}}}  u^{s-1-d_{h}/d_{w}} \exp\biggl(-c_{4}\Bigl(\frac{1}{u}\Bigr)^{\frac{1}{d_{w}-1}}\biggr)  du \\
  \le & d(x,y)^{sd_w-d_h} \int_0^{+\infty }  u^{s-1-d_{h}/d_{w}} \exp\biggl(-c_{4}\Bigl(\frac{1}{u}\Bigr)^{\frac{1}{d_{w}-1}}\biggr)  du.
\end{align*}
Finally, if $s =d_{h}/d_{w}$, using again  the change of variable $t= u d(x,y)^{d_{w}}$ and setting $R=\mathrm{ diam}(K)$, one sees that
\begin{align*}
 & \int_0^{1} t^{s-1-d_{h}/d_{w}}\exp\biggl(-c_{4}\Bigl(\frac{d(x,y)^{d_{w}}}{t}\Bigr)^{\frac{1}{d_{w}-1}}\biggr) dt \\
  = &  \int_0^{1/d(x,y)^{d_{w}}}  \frac{1}{u} \exp\biggl(-c_{4}\Bigl(\frac{1}{u}\Bigr)^{\frac{1}{d_{w}-1}}\biggr)  du \\
  \le &   \int_{1/R^{d_{w}}}^{1/d(x,y)^{d_{w}}}  \frac{1}{u} \exp\biggl(-c_{4}\Bigl(\frac{1}{u}\Bigr)^{\frac{1}{d_{w}-1}}\biggr)  du  +\int_0^{1/R^{d_{w}}}  \frac{1}{u} \exp\biggl(-c_{4}\Bigl(\frac{1}{u}\Bigr)^{\frac{1}{d_{w}-1}}\biggr)  du \\
  \le &   \int_{1/R^{d_{w}}}^{1/d(x,y)^{d_{w}}}  \frac{1}{u}   du  +\int_0^{1/R^{d_{w}}}  \frac{1}{u} \exp\biggl(-c_{4}\Bigl(\frac{1}{u}\Bigr)^{\frac{1}{d_{w}-1}}\biggr)  du \\
  \le &  d_w | \ln d(x,y) | +\int_0^{1/R^{d_{w}}} \frac{1}{u} \exp\biggl(-c_{4}\Bigl(\frac{1}{u}\Bigr)^{\frac{1}{d_{w}-1}}\biggr)  du \le C | \ln d(x,y) |.
\end{align*}
\end{proof}

\begin{proposition}\label{integrability G}
If $s >\frac{d_h}{2d_w}$, then for every $x \in K$, $y \to G_s(x,y) \,  \in L_0^2(K,\mu)$. Moreover, there exists a constant $C>0$ such that for every  $x \in K$,
\[
\int_K G_s(x,y)^2 d\mu(y) \le C.
\]
\end{proposition}

\begin{proof}
From proposition \ref{estimate G}, it is enough to prove that for $\gamma  <\frac{d_h}{2}$, the function $y \to \frac{C}{d(x,y)^{\gamma }}$ is in $L^2(K,\mu)$ (since e.g. for $\alpha>0$, $\max(1,|\ln u|)\le \frac{C}{u^\alpha}$ for $0<u\le u_0$). To prove this, we denote by $R$ the diameter of $K$ and use  the Ahlfors regularity \eqref{Ahlfors} of the measure $\mu$ and a dyadic  annuli decomposition as follows. We denote by $C$ constants (depending only on $R,s,d_h,d_w$) whose value may change from line to line. One has: for $\gamma<\frac{d_h}{2}$,
\begin{align*}
 \int_K \frac{d\mu(y)}{d(x,y)^{2\gamma}} 
 & \le \sum_{j=0}^{+\infty} \int_{B(x,R 2^{-j})\setminus B(x,R 2^{-j-1}) } \frac{d\mu(y)}{d(x,y)^{2\gamma}} \\
 & \le C \sum_{j=0}^{+\infty} 2^{2 j \gamma} \mu \left( B(x,R 2^{-j})\setminus B(x,R 2^{-j-1}) \right) \\
 & \le C \sum_{j=0}^{+\infty} 2^{ 2j \gamma} \mu \left( B(x,R 2^{-j}) \right) \\
 &\le C  \sum_{j=0}^{+\infty} 2^{ j (2\gamma-d_h)}  <+\infty,
\end{align*}
which concludes the proof.
\end{proof}

\begin{proposition}\label{relation to frac}
Let $s >\frac{d_h}{2d_w}$ and consider the operator $\mathcal{G}_s : L_0^2(K,\mu) \to L^2_0(K,\mu) $ defined  by
\[
\mathcal{G}_s f (x)= \int_K G_s(x,y) f(y) d\mu(y), \, x \in K.
\]
Then for every $f \in L_0^2(K,\mu)$, one has $\mu$ a.e.
\[
(-\Delta)^{-s}f=\mathcal{G}_s f.
\]
\end{proposition}
\begin{remark}\label{continuous version}
It is important to note that from proposition \ref{integrability G},  $\mathcal{G}_s f $ is defined for all $x\in K$ and not only $\mu$ a.e. Therefore $\mathcal{G}_s f$ can be used as a pointwise definition of $(-\Delta)^{-s}f$.
\end{remark}
\begin{proof}
Let $f \in L_0^2(K,\mu)$. One can write
\[
f=\sum_{j=1}^{+\infty}  \left( \int_K  \Phi_j(y) f(y) d\mu(y)\right) \Phi_j
\]
where the sum is convergent in $ L_0^2(K,\mu)$. From proposition \ref{integrability G} the operator $\mathcal{G}_s : L_0^2(K,\mu) \to L^2_0(K,\mu) $ is bounded. 

Therefore, in $L_0^2(K,\mu)$
\[
\mathcal{G}_s f=\sum_{j=1}^{+\infty}  \left( \int_K  \Phi_j(y) f(y) d\mu(y)\right) \mathcal{G}_s \Phi_j.
\]
By definition of $\mathcal{G}_s$, we  now compute that for  $x \in K$
\begin{align*}
\mathcal{G}_s \Phi_j(x)& = \int_K G_s(x,y) \Phi_j(y) d\mu(y) \\
 & =\frac{1}{\Gamma(s)} \int_K  \int_0^{+\infty} t^{s-1} (p_t(x,y) -1) \Phi_j(y) dt d\mu(y) \\
 & =\frac{1}{\Gamma(s)}   \int_0^{+\infty} t^{s-1} \int_K  (p_t(x,y) -1) \Phi_j(y)  d\mu(y) dt \\
 & =\frac{1}{\Gamma(s)}   \int_0^{+\infty} t^{s-1} \int_K  p_t(x,y)  \Phi_j(y)  d\mu(y) dt \\
&  =\frac{1}{\Gamma(s)}   \int_0^{+\infty} t^{s-1} (P_t  \Phi_j)(x)   dt \\
& =\frac{1}{\Gamma(s)}   \int_0^{+\infty} t^{s-1} e^{-\lambda_j t}    dt \Phi_j (x) \\
&= \lambda_j^{-s} \Phi_j (x).
\end{align*}
Therefore, one has $\mu$ a.e.
\[
\mathcal{G}_s f =\sum_{j=1}^{+\infty} \frac{1}{\lambda_j^s} \left( \int_K  \Phi_j(y) f(y) d\mu(y)\right) \Phi_j=(-\Delta)^{-s}f ,
\]
which establishes the proof.
\end{proof}

\subsection{H\"older continuity of fractional Riesz kernels}
The main theorem of the section is the following:

\begin{theorem}\label{Holder Riesz}
Let $s \in \left( \frac{d_h}{2d_w} , 1- \frac{d_h}{2d_w}\right)$. There exists a constant $C>0$ such that for every  $x,y \in K$ and $ f \in L^2(K,\mu)$,
\[
\left| \int_K (G_s (x,z)-G_s(y,z))  f(z) d\mu(z) \right|\le  Cd(x,y)^{sd_w-\frac{d_h}{2}} \| f \|_{L^2(K,\mu)}.
\]
As a consequence, there exists a constant $C>0$ such that for every  $x,y \in K$,

\[
\int_X (G_s (x,z)-G_s(y,z))^2  d\mu(z) \le Cd(x,y)^{ 2sd_w-d_h} .
\]
\end{theorem}

We divide the proof in several lemmas. As usual, we will denote by $C$ constants whose value may change from line to line.

\begin{lemma}
	\label{contPt1}
There exists a constant $C>0$ such that for every $f \in L^2(K,\mu)$, $t > 0$ and $x \in K$,
\[
| P_t f(x) | \le  \frac{C}{t^{\frac{d_h}{2d_w}}} \| f \|_{L^2(K,\mu)}.
\]

\end{lemma}
\begin{proof}
From Cauchy-Schwarz inequality,
\begin{align*}
| P_t f(x) |^2 &=\left| \int_K p_t (x,z) f(z) d\mu(z) \right|^2 \\
 & \le \int_K p_t (x,z)^2  d\mu(z) \| f \|^2_{L^2(K,\mu)} \\
 & \le p_{2t} (x,x) \| f \|^2_{L^2(K,\mu)}.
\end{align*}
We conclude then with the sub-Gaussian upper bound \eqref{eq:subGauss-upper}.
\end{proof}

\begin{lemma}
	\label{contPt2}
There exists a constant $C>0$ such that for every  $f \in L^2(K,\mu)$, $t > 0$ and  $x,y \in K$,
\[
| P_t f(x) -P_tf(y) | \le C \frac{d(x,y)^{d_w-d_h}}{t^{1 - \frac{d_h}{2d_w}}} \| f \|_{L^2(K,\mu)}.
\]
\end{lemma}

\begin{proof}
From  \cite{ABCRST3,BAR}, it is known that  for the Sierpi\'nski gasket there exists a constant $C>0$ such that for every $g \in L^\infty(K,\mu)$, $t > 0$ and $x,y \in K$,
\[
| P_t g(x) -P_tg(y) | \le C \frac{d(x,y)^{d_w-d_h}}{t^{1 - \frac{d_h}{d_w}}} \| g \|_{L^\infty(K,\mu)}.
\]
Now, if $f \in L^2(K,\mu)$, then from the previous lemma $P_t f\in L^\infty (K,\mu)$, so that the previous inequality can be applied to $g=P_tf$. Using the semigroup property this yields
\[
| P_{2t} f(x) -P_{2t}f(y) | \le C \frac{d(x,y)^{d_w-d_h}}{t^{1 - \frac{d_h}{2d_w}}} \| f \|_{L^2(K,\mu)},
\]
which concludes the proof.
\end{proof}

Our third lemma is the following:
\begin{lemma}
	\label{contPt3}
Let $\frac{d_h}{2d_w} < s < 1-\frac{d_h}{2d_w}$. There exists a constant $C>0$ such that for every $f \in L^2(K,\mu)$ and  $x,y \in K$,
\[
\int_0^{+\infty} t^{s-1} | P_t f (x) -P_tf(y)| dt \le Cd(x,y)^{sd_w-\frac{d_h}{2}} \| f \|_{L^2(K,\mu)}.
\]
\end{lemma}

\begin{proof}
We split the integral into two parts:
\[
\int_0^{+\infty} t^{s-1} | P_t f (x) -P_tf(y)| dt =\int_0^{\delta} t^{s-1} | P_t f (x) -P_tf(y)| dt+\int_\delta^{+\infty} t^{s-1} | P_t f (x) -P_tf(y)| dt
\]
where $\delta >0$ will later be optimized.  First, applying lemma \ref{contPt1}, we have
\begin{align*}
\int_0^{\delta} t^{s-1} | P_t f (x) -P_tf(y)| dt &\le \int_0^{\delta} t^{s-1}( | P_t f (x)| +|P_tf(y)| )dt \\
 & \le \int_0^{\delta} t^{s-1} \frac{C}{t^{\frac{d_h}{2d_w}}} dt \| f \|_{L^2(K,\mu)} \\
 &\le C \delta^{ s- \frac{d_h}{2d_w}} \| f \|_{L^2(K,\mu)}.
\end{align*}
Then, applying lemma \ref{contPt2}, we have
\begin{align*}
\int_\delta^{+\infty} t^{s-1}  | P_t f (x) -P_tf(y)| dt &\le C \int_\delta^{+\infty} t^{s-1} \frac{d(x,y)^{d_w-d_h}}{t^{1 - \frac{d_h}{2d_w}}}  \| f \|_{L^2(K,\mu)}dt \\
 & \le  C d(x,y)^{d_w-d_h} \int_\delta^{+\infty} t^{s-2+ \frac{d_h}{2d_w}} dt \| f \|_{L^2(K,\mu)} \\
 &\le C  d(x,y)^{d_w-d_h} \delta^{s-1+ \frac{d_h}{2d_w}} \| f \|_{L^2(K,\mu)}.
\end{align*}
One concludes
\[
\int_0^{+\infty} t^{s-1} | P_t f (x) -P_tf(y)| dt \le C \left(\delta^{ s- \frac{d_h}{2d_w}}  + d(x,y)^{d_w-d_h} \delta^{s-1+ \frac{d_h}{2d_w}}  \right)\| f \|_{L^2(K,\mu)}.
\]
Choosing then $\delta=d(x,y)^{d_w}$ yields the expected result.
\end{proof}

We are finally ready for the proof of the main theorem:

\begin{proof}
One has
\begin{align*}
\left| \int_K (G_s (x,z)-G_s(y,z))  f(z) d\mu(z) \right|&= C \left|  \int_0^{+\infty} t^{s-1} (P_tf(x)-P_tf(y))  dt \right| \\
 &\le C  \int_0^{+\infty} t^{s-1}  \left|  P_tf(x) -P_tf(y) \right|  dt \\
 & \le  Cd(x,y)^{sd_w-\frac{d_h}{2}} \| f \|_{L^2(K,\mu)}.
\end{align*}
By $L^2$ self-duality, one concludes
\[
\int_K (G_s (x,z)-G_s(y,z))^2  d\mu(z) \le Cd(x,y)^{2sd_w-d_h}.
\]
\end{proof}

\section{Fractional Brownian fields on the gasket}

\subsection{Reminders on Gaussian measures}

Given a probability space  $\left(\Omega,\mathcal{F},\mathbb{P}\right)$, we consider on the measurable space $(K,\mathcal{K},\mu)$, where $\mathcal{K}$ is the Borel $\sigma$-field on $K$,  a  real-valued Gaussian random measure $W_K\,:\mathcal{K}\rightarrow L^2\left(\Omega,\mathcal{F},\mathbb{P}\right)$ with intensity $\mu$. In other words, $W_K$ is such that 
\begin{itemize}
	\item a.s. $W_K$ is a measure on $(K,\mathcal{K})$ 
	\item for any $A\in\mathcal{K}$ such that $\mu(A)<\infty$,   $W_K(A)$ is a real-valued Gaussian variable with mean zero and variance $\mathbb{E}\left(W_K\left(A\right)^2\right)=\mu(A)$
	\item for any sequence $(A_n)_{n\in\mathbb{N}} \in \mathcal{K}^\mathbb{N}$ of pairwise disjoint measurable sets, the random variables $W_K(A_n)$, $n\in\mathbb{N}$, are independent. 
\end{itemize}
Then for any $f\in L^2(K,\mathcal{K},\mu)$, the stochastic integral 
$$
W_K(f)=\int_{K} f  \,dW_K
$$
is well-defined and is a centered real-valued Gaussian variable, see e.g. \cite[Section 2.3]{MR3466837} for details on the construction. Moreover, denoting by   $H\subset L^2\left(\Omega,\mathcal{F},\mathbb{P}\right)$ the Gaussian Hilbert space spanned by 
$
\{ W_K(A); A\in\mathcal{K}, \mu(A)<\infty\}
$, the functional $W_K:  L^2(K,\mathcal{K},\mu)\rightarrow H$ is an isometry. Hence, for any $f,g\in L^2(K,\mathcal{K},\mu)$, 
\begin{equation}
\label{isometry-property}
\mathbb{E}\left(\int_K f \,dW_K \int_{K} g \,dW_K \right)= \langle f,g\rangle_{L^2(K,\mathcal{K},\mu)}=\int_{K} f g\,d\mu.
\end{equation}

\subsection{Definition and existence of the fractional Brownian field}

\begin{definition}[Fractional Brownian field with parameter $H$]
Let $H\in (0,d_w-d_h)$. We define the fractional Brownian field with  parameter $H$ as the random field given by
\[
X(x)=\int_K  G_s(x,z) \, W_{K} (dz), \, x\in K,
\]
where $s= \frac{d_h+2H}{2d_w}$,  $W_{K}$ is a Gaussian centered real-valued random measure on $L_0^2(K,\mu)$ with intensity $\mu$ and $G_s$ is the Riesz kernel  defined by \eqref{green function}.
 \end{definition}
 \begin{remark}
  Thanks to proposition \ref{integrability G}, the random variable $X(x)$ is well defined for all $x \in K$.
 \end{remark}
 
 \begin{remark}
 Thanks to proposition \ref{relation to frac},  one has for every $f$ which is in the $L^2_0$ domain of the operator $(-\Delta )^s  $
\[
\int_{K} (-\Delta )^s f (x) X(x)  d\mu(x)=\int_{K}  f (x) W_K (dx).
\]
 \end{remark}
 \begin{remark}
 The Gaussian field $(X(x))_{x \in K}$ has mean zero and covariance
 \[
 \mathbb{E}( X(x) X(y)) = \int_K  G_s(x,z) G_s (y,z) d\mu(z)=G_{2s} (x,y).
 \]
 We note that since $2s > d_h/d_w$, from proposition \ref{estimate G}, the function $G_{2s}$ is uniformly bounded on $K \times K$.
 \end{remark}
 \begin{remark}
 One could also consider the  random field given by
\[
\tilde{X}(x)=\int_K  (G_s(x,z)-G_s(q,z)) \, W_{K} (dz).
\]
where $q \in K$ is an arbitrary point of the gasket.
 \end{remark}

\begin{theorem}\label{eq:contvar}
Let $H\in (0,d_w-d_h)$, then  there exists a constant $C>0$ so that for every  $x,y \in K$,
\[
\mathbb{E} ((X(x)-X(y))^2)\le C d(x,y)^{2  H }.
\]
\end{theorem}

\begin{proof}
Since
\[
\mathbb{E} ((X(x)-X(y))^2)=\int_K  (G_s(x,z)-G_s(y,z))^2 d\mu(z)
\]
this follows from Theorem \ref{Holder Riesz}.
\end{proof}

\begin{proposition}
Let $H\in (0,d_w-d_h)$, then the fractional Brownian field $(X(x))_{x \in K}$ with parameter $H$ admits   a spectral expansion
\[
X= \sum_{j=1}^{+\infty} \frac{1}{\lambda_j^s}    N_j \,  \Phi_j
\]
where the $N_i$'s are i.i.d. normal centered Gaussian random variables with variance 1 and the series is convergent in $L^2(K \times \Omega ,\mu \otimes \mathbb{P} )$. 
\end{proposition}

\begin{proof}
Note that from the expansion \eqref{spectral}, one obtains that $\mu \otimes \mu$ a.e. $x,y \in K$
\[
G_s(x,y)= \sum_{j=1}^{+\infty} \frac{1}{\lambda_j^s} \Phi_j(x) \Phi_j(y)
\]
where the sum on the right hand side is convergent in $L^2(K\times K, \mu \otimes \mu)$. Since the $ \Phi_j$'s form a complete orthonormal system in $L^2_0(K,\mu)$, one easily proves that 
\[
N_j =\int_K  \Phi_j(z) \, W_{K} (dz)
\]
 $N_i$'s are i.i.d. normal centered Gaussian random variables with variance 1.
\end{proof}

\subsection{Regularity of the fractional Brownian field}

Barlow and Perkins have established the following  Garsia-Rodemich-Rumsey inequality for fractal (see lemma 6.1 of \cite{BP88}). 
\begin{lemma}\label{BPlem} Let 
	$p$ be an increasing function on $[0,\infty)$ with $p(0)=0,$ and $\psi : \mathbb{R}\rightarrow  \mathbb{R}_+$ be a non-negative symmetric convex function with $\lim_{u\to \infty} \psi(u)=\infty$. Let $f: K\rightarrow  \mathbb{R}$ be a measurable function such that 
	$$
	\Gamma =\int_{K\times K} \psi \left(\frac{\left|f(x)-f(y)\right|}{p(d(x,y))}\right)\mu({\rm d}x)\mu({\rm d}y)
	<+\infty$$
	Then there exists a constant $c_K$ depending only on  $d_h$ such that 
	$$
	\left|f(x)-f(y)\right|\le 8\int_0^{d(x,y)} \psi^{-1}\left(\frac{c_K \Gamma}{u^{2d_h}}\right) \,p({\rm d}u)
	$$
	for $\mu\times \mu$ almost all $x,y\in K\times K$. 
	\end{lemma}

The main result of this section is the following.
\begin{theorem}\label{regularite}
	There exists a modification $X^*$ of $X$ such that  
	$$
	\lim_{\delta \to 0}\underset{\underset{x,y\in K}{\scriptsize d(x,y)}\le \delta}{\sup}\  \frac{\left| X^*(x)-X^*(y)\right|}{d(x,y)^H \sqrt{\left|\ln d(x,y)\right|}} <+\infty	$$ 
	\end{theorem}  
	
	\begin{proof}

		{\bf Step 1: Control of $\boldsymbol{X(x)-X(y)}$ for a.e  $\boldsymbol{\omega}$ and for $\boldsymbol{\mu}$ a.e. $\boldsymbol{x,y\in K}$}
		
		Let us consider $p(u)=u^H$ and $\psi(u)=\exp(\frac{u^2}{4c^2})-1$ where $c>0$ will be chosen later.   
Let 
$$
\Gamma(\omega) =\int_{K\times K} \psi \left(\frac{\left|X(x,\omega)-X(y,\omega)\right|}{p(d(x,y))}\right)\mu({\rm d}x)\mu({\rm d}y).
$$
Then by Fubini Theorem, 
$$
\mathbb{E}\left(\Gamma\right)=\int_{K\times K} \mathbb{E}\left(\psi \left(\frac{\left|X(x)-X(y)\right|}{p(d(x,y))}\right)\right)\mu({\rm d}x)\mu({\rm d}y)
$$		
Let us consider $x,y\in K$ such that $x\ne y$.	Since $\psi$ is non-negative,
	$$
	\mathbb{E}\left(\psi \left(\frac{\left|X(x)-X(y)\right|}{p(d(x,y))}\right)\right)=\int_{0}^{\infty} \mathbb{P}\left(\psi\left(\frac{\left|X(x)-X(y)\right|}{p(d(x,y))}\right)>t\right) {\rm d}t 
	$$	
	so that by definition of $\psi$, 
$$
\mathbb{E}\left(\psi \left(\frac{\left|X(x)-X(y)\right|}{p(d(x,y))}\right)\right)=\int_1^\infty \mathbb{P}\left(\left|X(x)-X(y)\right|\ge 2c\sqrt{\log(t)}\, p(d(x,y))\right) {\rm d}t 
$$
Hence, by  Theorem  \ref{eq:contvar},  there exists a finite positive constant $C$ such that  for every  $x,y\in K$
$$
 p(d(x,y))= {d(x,y)^{H }}\ge \frac{\sqrt{\mathbb{E}(\left(X(x)-X(y)\right)^2}}{C}.
$$
Hence choosing $c=C$, we obtain that for every $x,y\in K$ 
\begin{align*}
 & \mathbb{E}\left(\psi \left(\frac{\left|X(x)-X(y)\right|}{p(d(x,y))}\right)\right) \\
\le & \int_1^\infty \mathbb{P}\left(\left|X(x)-X(y)\right|\ge 2\sqrt{\log(t)}\, \sqrt{\mathbb{E}(\left(X(x)-X(y)\right)^2)}\right) {\rm d}t 
\end{align*}
Let us now recall that  for $Z$ a standard Gaussian random variable,
$$\forall \lambda>0,\ 
\mathbb{P}\left(|Z|\ge \lambda \right)\le \sqrt{\frac{2}{\pi}}\frac{\exp\left(-\frac{\lambda^2}{2}\right)}{\lambda}
$$
Therefore since $X(x)-X(y)$ is a centered Gaussian random variable with variance $\mathbb{E}(\left(X(x)-X(y)\right)^2)\ne 0$, 
$$
\displaystyle \mathbb{E}\left(\psi \left(\frac{\left|X(x)-X(y)\right|}{p(d(x,y))}\right)\right)\le\displaystyle  \textup{e}+\frac{1}{\sqrt{2\pi}}\int_{\textup{e}}^{\infty} \frac{1}{t^2 \sqrt{\log(t)}}{\rm d}t\le  \displaystyle \textup{e} +\frac{1}{\sqrt{2\pi}}
$$
for every  $x,y\in K$ such that $x\ne y$.  
Then 
$$
\mathbb{E}\left(\Gamma\right) \le \left( \textup{e} +\frac{1}{\sqrt{\pi}}\right)\mu(K)^2<\infty.
$$
Hence, there exists $\tilde{\Omega}$ such that $\mathbb{P}(\tilde{\Omega})=1$ and for all $\omega\in\tilde{\Omega}$ $\Gamma(\omega)<\infty$. By applying Lemma \ref{BPlem},  for all $\omega \in \tilde{\Omega}$,  there exists a finite constant $ c_K$ which only depends on  $d_h$ and  $J(\omega)\subset K$ with $\mu(K\backslash J(\omega))=0$ such that  
$$
\left|X(x)-X(y)\right|\le 8\int_0^{d(x,y)} \psi^{-1}\left(\frac{c_K \Gamma}{u^{2d_h}}\right) \,p({\rm d}u)=16 c\int_0^{d(x,y)} \sqrt{\log\left(1+\frac{c_K \Gamma}{u^{2d_h}}\right)} \,p({\rm d}u)
$$
for every $\omega \in\tilde{\Omega}$ and $x,y\in J(\omega)$. Then there exists a finite positive random variable $A$ such that for every $\omega \in\tilde{\Omega}$ and $x,y\in J(\omega)$, 
$$
\left|X(x)-X(y)\right|\le A\int_0^{d(x,y)} \sqrt{|\log(u)|} \, p({\rm d}u)
$$ 
Hence for $x,y\in J(\omega)$ such that $d(x,y)\le 1/\textup{e}$, using integration by parts, up to change $A$ at each line, 
$$\begin{array}{rcl}
\displaystyle \left|X(x)-X(y)\right|&\le& \displaystyle  A d(x,y)^{H}\sqrt{-\log(d(x,y))} +A\int_0^{d(x,y)}\frac{u^{H-1}}{\sqrt{-\log(u)}} \, {\rm d}u\\[15pt]
&\le & \displaystyle Ad(x,y)^{H}\sqrt{-\log(d(x,y))}+Ad(x,y)^{H}.\end{array}
$$
Therefore, setting $\varphi(u)=u^{H}\sqrt{-\log(u)}$ for $u>0$, 
\begin{equation}
\label{eq:controlae}
\displaystyle \left|X(x)-X(y)\right|
\le\displaystyle A  \varphi(d(x,y))
\end{equation}
for all $x,y\in J(\omega)$ such that $d(x,y)\le 1/\textup{e}$.  
 {Since $H>0$, let us remark that  we can choose $r\in (0,1/\textup{e})$ such that $\varphi$ is a non-decreasing function on $(0,r)$.}\\

{\bf Step 2: Definition of $\boldsymbol{X^*}$}

Let us consider $\omega\in \tilde{\Omega}$ and $x\in K$.  For any $n\in\mathbb{N}^*$, let $B_n(x)=B\left(x,\frac{1}{n}\right)=\{y\in K\,/\, d(x,y)\le \frac{1}{n}\}$ and 
$$
X_n(x,\omega)=\frac{1}{\mu\left(B_n\left(x\right)\right)}\int_{B_n\left(x\right)}X(u,\omega)\,\mu({\rm d}u). 
$$
Note that for any integer $n,m$, 
\begin{align*}
& X_n(x,\omega)-X_{m}(x,\omega) \\
=&\frac{1}{\mu\left(B_n\left(x\right)\right)\mu\left(B_m\left(x\right)\right)}\int_{B_n\left(x\right)}\int_{B_m\left(x\right)}\left(X(u,\omega)-X(v,\omega)\right)\,\mu({\rm d}u)\mu({\rm d}v).
\end{align*}

Since $\mu\left(K\backslash J(\omega)\right)=0$ and since $\varphi$ is a non-decreasing function on $(0,r)$, by applying \eqref{eq:controlae}, we have:   
$$
\left|X_n(x,\omega)-X_{m}(x,\omega)  \right|\le A\varphi\left(\frac{2}{n}\right)
$$
for any $n>2/r  $ and any  $m\ge n$. 
Since $\lim_{r\to 0_+}\varphi(r)=0$,  $(X_n(x,\omega))_n$ is a real-valued Cauchy sequence and then converges. Then for $\omega\in\tilde{\Omega}$ and $x\in K$, we define 
$$
X^*(x,\omega):=\lim_{n\to +\infty} X_n(x,\omega).
$$
If $\omega\notin \tilde{\Omega}$, we set  $X^*(x,\omega)=0$ for every $x\in K$. \\

{\bf Step 3 : Upper bound for the modulus of continuity of $\boldsymbol{X^*}$} 

Let us first assume that $\omega\in \tilde{\Omega}$. Then for any $x,y\in K$, 
\begin{align*}
 & X_n(x,\omega)-X_{n}(y,\omega) \\
 =&\frac{1}{\mu\left(B_n\left(x\right)\right)\mu\left(B_n\left(y\right)\right)}\int_{B_n\left(x\right)}\int_{B_n\left(y\right)}\left(X(u,\omega)-X(v,\omega)\right)\,\mu({\rm d}u)\mu({\rm d}v)
\end{align*}
Let $x,y\in K$ such that $0<d(x,y)<r/2$. Then for $u\in B_n(x)$ and $v\in B_n(x)$, if $n>4/r$, 
$$
d(u,v)\le d(x,y)+\frac{2}{n}<r
$$
Hence,  for $n>4/r$, 
$$
\varphi\left(d(u,v)\right)\le \varphi\left(d(x,y)+\frac{2}{n}\right)
$$
and, since $\mu\left(K\backslash J(\omega)\right)=0$,  by \eqref{eq:controlae}, 
$$
\left|X_n(x,\omega)-X_{n}(y,\omega)\right|\le A \varphi\left(d(x,y)+\frac{2}{n}\right).
$$
 Letting $n\to +\infty$,  by continuity of $\varphi$ on $(0,r)$, we obtain that  
 $$
 \left|X^*(x,\omega)-X^*(y,\omega)\right|\le A \varphi\left(d(x,y)\right).
 $$
 for every $x,y\in K$ such that $d(x,y)\le r/2$. 
 Note that this last inequality also holds if $\omega\notin \tilde{\Omega}$.  
\\

{\bf Step 4: Comparison of $\boldsymbol{X}$ and $\boldsymbol{X^*}$ }

Let $x\in K$. Note that 
$$
X_n(x)-X(x)=\frac{1}{\mu\left(B_n(x)\right)}\int_{B_n(x)}\left(X(u,\omega)-X(x,\omega)\right)\mu({\rm d}u)
$$
Therefore, by Cauchy-Schwarz inequality, 
$$
\left|X_n(x)-X(x)\right|^2\le \frac{1}{\mu\left(B_n(x)\right)}\int_{B_n(x)}\left|X(u,\omega)-X(x,\omega)\right|^2\mu({\rm d}u)
 $$
Hence, applying Fubini theorem,
 $$
\mathbb{E}\left( \left|X_n(x)-X(x)\right|^2\right)\le\frac{1}{\mu\left(B_n(x)\right)} \int_{B_n(x)}\mathbb{E}\left(\left|X(u,\omega)-X(x,\omega)\right|^2\right)\mu({\rm d}u)
 $$
 By Theorem \ref{eq:contvar}, for every $x\in K$, and all $n\in\mathbb{N}^*$, 
 $$\begin{array}{rcl}
\displaystyle  \mathbb{E}\left( \left|X_n(x)-X(x)\right|^2\right)&\le&\displaystyle \frac{C}{\mu\left(B_n(x)\right)} \int_{B_n(x)}d(u,x)^{2H}\mu({\rm d}u)\\[15pt]
&\le & \frac{C}{n^{2H}}.
\end{array}
  $$
   Then for every $x\in K$, $(X_n(x))_n $ converges in quadratic mean  to $X(x)$, which implies that  $X(x)=X^*(x)$ a.s. In other words, $X^*$ is a modification of $X$, which concludes the proof.  
		\end{proof}

\subsection{Invariance and scaling properties of the fractional Brownian field}

\subsubsection{Invariance by symmetries}

The Sierpi\'nski gasket admits 3 symmetries $\sigma_1,\sigma_2,\sigma_3$ which are the reflections about the lines dividing the triangle with vertices $q_0,q_1,q_2$ into two equal parts.

\begin{proposition}
 Let $H\in (0,d_w-d_h)$ and
\[
X(x)=\int_K  G_s(x,z) \, W_{K} (dz), \, x\in K,
\]
be the fractional Brownian field with parameter $H$. Then, for every $i=1,2,3$ in distribution
\[
(X(\sigma_i(x))_{x \in K}=^{d} (X(x))_{x \in K}.
\]
\end{proposition}

\begin{proof}
The Dirichlet form $\mathcal E$ on the gasket is invariant by $\sigma_i$, i.e. for every $ f\in \mathcal F$
\[
\mathcal{E}( f\circ \sigma_i, f\circ \sigma_i)=\mathcal{E}( f, f).
\]
Thus, for every $x,y \in K$, $p_t( \sigma_i(x), \sigma_i(y))= p_t(x,y)$. This implies that $G_{2s}( \sigma_i(x), \sigma_i(y))= G_{2s}(x,y)$ and thus $\mathbb{E}( X(\sigma_i(x)) X(\sigma_i(y)) =\mathbb{E}(X(x)X(y))$.
\end{proof}

In particular, at the vertices, one obtains that $X(q_0),X(q_1),X(q_2)$ have the same distribution.
\subsubsection{Invariance by scaling}

Let $w=(i_1, \cdots, i_n) \in \{ 1,2,3 \}^n $, and denote $F_w=F_{i_1} \circ \cdots \circ F_{i_n}$ where we recall that 
\[
F_i(z)=\frac{1}{2}(z-q_i)+q_i.
\]
The compact set $K_w:=F_w(K) \subset K$ is itself a Sierpi\'nski gasket. Denote by $X^w$ a fractional Brownian motion field with parameter $H$ on $K_w$.

\begin{proposition}
The Gaussian field $(2^{nH} X^w( F_w (x)))_{x \in K}$ is a fractional Brownian motion field with parameter $H$ on $K$.
\end{proposition}

\begin{proof}
In the proof let us indicate with a superscript or subscript $w$ the objects related to the Sierpi\'nski gasket $K_w$ (Dirichlet form, heat kernel, etc...).
From the limit \eqref{dirichlet limit}, one can see that for every $f \in \mathcal{F}_w$, one has
\[
\mathcal{E}_w(f,f)=\left( \frac{5}{3} \right)^n \mathcal{E}(f \circ F_w ,f\circ F_w).
\]
Thus the relation between the Laplacian of $K_w$ and the Laplacian of $K$ is given
\[
(\Delta_w f )\circ F_w= 5^n\Delta (f \circ F_w).
\]
This yields that for the heat kernels (with respect to the reference measure $\mu$)
\[
p_t^w( F_w(x),F_w(y))=3^n p_{5^n t} (x,y).
\]
As a consequence, one has for $ x,y \in K, \, x\neq y$, 
\begin{align*}
G_s^w(F_w(x),F_w(y)) &=\frac{1}{\Gamma(s)} \int_0^{+\infty} t^{s-1} (p^w_t(F_w(x),F_w(y)) -1)dt \\
 &=\frac{3^n}{\Gamma(s)} \int_0^{+\infty} t^{s-1} (p_{5^n t} (x,y) -1)dt \\
 &=\frac{3^n}{5^{ns}\Gamma(s)} \int_0^{+\infty} t^{s-1} (p_{ t} (x,y) -1)dt .
\end{align*}
Since
\[
s=\frac{d_h}{2d_w}+\frac{H}{d_w}=\frac{1}{2} \frac{\ln 3}{\ln 5} +H \frac{\ln 2}{\ln 5}
\]
one has $5^{ns}=2^{nH} 3^{n/2}$ and therefore
\[
G_s^w(F_w(x),F_w(y)) =\frac{3^{n/2}}{2^{nH}} G_s(x,y).
\]
Notice now that if $W_{K_w}$ is a white noise on $L_0^2(K_w,\mu)$, due to the self-similarity of the Hausdorff measure $\mu$ one has for every $f \in L_0^2(K_w,\mu)$,
\begin{align*}
\mathbb{E} \left( \left(  \int_{K_w} f(z) W_{K_w}(dz)  \right)^2 \right)&=\int_{K_w} f(z)^2 d\mu(z) \\
 & =\frac{1}{3^n} \int_{K} f(F_w(z))^2 d\mu(z) \\
 &=\frac{1}{3^n} \mathbb{E} \left( \left(  \int_{K} f(F_w(z)) W_{K}(dz)  \right)^2 \right).
\end{align*}
One concludes that in distribution:
\begin{align*}
X(F_w(x))&=\int_{K_w}  G^w_s(F_w(x),z) \, W_{K_w} (dz) \\
 &=\frac{1}{3^{n/2}} \int_{K}  G^w_s(F_w(x),F_w(z)) \, W_{K} (dz) \\
 &=\frac{1}{2^{nH}} \int_{K}  G_s(x,z) \, W_{K} (dz).
\end{align*}
\end{proof}

\section{Generalization to other fractals: Barlow fractional spaces}

\begin{figure}[htb]\centering
 	\includegraphics[trim={60 10 180 60},height=0.25\textwidth]{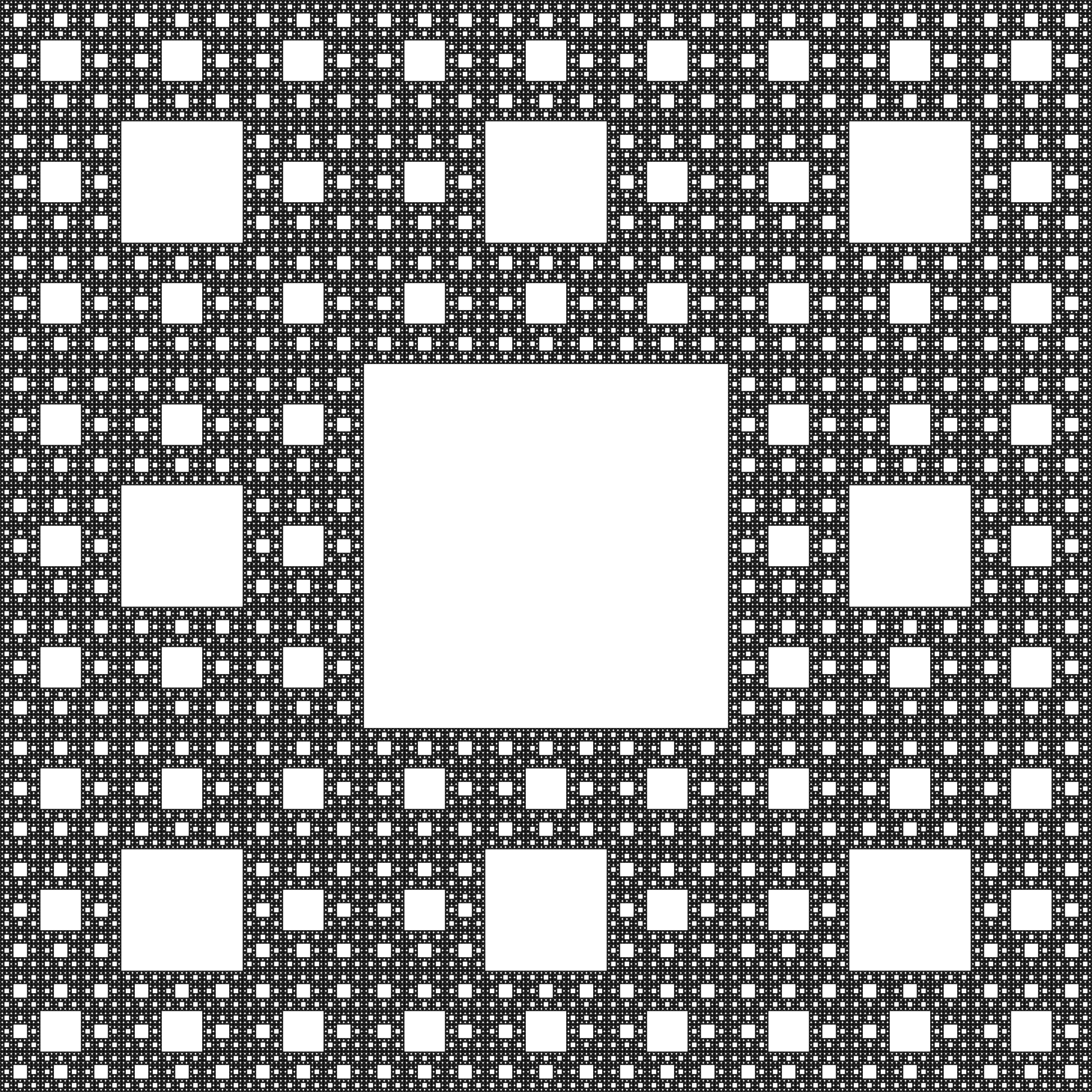}
 	\caption{Sierpi\'nski carpet.} \label{fig-SC}
 \end{figure}

Let $(K,d,\mu)$ be a compact metric  space isometrically embedded in some Euclidean space where $\mu$ is the Hausdorff measure on $K$. Let now $(\mathcal{E},\mathcal{F}=\mathbf{dom}(\mathcal{E}))$ be a strongly local regular Dirichlet form on $L^2(K,\mu)$. 

We  assume that the semigroup $\{P_t\}$ has a jointly continuous heat kernel $p_t(x,y)$ satisfying, for some $c_{1},c_{2}, c_3, c_4 \in(0,\infty)$ and $ d_h \ge 1, d_{w}\in [2,+\infty)$, $d_w \ge d_h$
 \begin{equation}\label{eq:subGauss-upper2}
 c_{1}t^{-d_{h}/d_{w}}\exp\biggl(-c_{2}\Bigl(\frac{d(x,y)^{d_{w}}}{t}\Bigr)^{\frac{1}{d_{w}-1}}\biggr) 
 \le p_{t}(x,y)\leq c_{3}t^{-d_{h}/d_{w}}\exp\biggl(-c_{4}\Bigl(\frac{d(x,y)^{d_{w}}}{t}\Bigr)^{\frac{1}{d_{w}-1}}\biggr)
 \end{equation}
 for every \ $(x,y)\in K\times K$ and  $t\in\bigl(0,1)$. 
 
 We moreover assume that metric space $(K,d)$ satisfies the midpoint property, i.e. for any $x, y \in K$ there exists $z \in K$ such that $d(x, z) = d(z, y) = \frac{1}{2} d(x, y)$. The latter is equivalent to requiring the space be geodesic.  Metric spaces satisfying the above assumptions are called fractional metric spaces and were extensively studied by Barlow  in Section 3 of the lectures \cite{Bar98}. Besides the Sierpi\'nski gasket studied previously, another popular fractal set that fits into this framework is the Sierpi\'nski carpet represented in Figure \ref{fig-SC}.
 
From  \cite{ABCRST3,BAR}, it is known that under the previous assumptions the measure $\mu$ is $d_h$-Ahlfors regular and that there exists a constant $C>0$ such that for every $f \in L^\infty(K,\mu)$, $t > 0$ and $x,y \in K$,
\begin{equation}\label{holder semigroup 2}
| P_t f(x) -P_tf(y) | \le C \frac{d(x,y)^{d_w-d_h}}{t^{1 - \frac{d_h}{d_w}}} \| f \|_{L^\infty(K,\mu)}.
\end{equation}

 For the Sierpi\'nski carpet it is known that $d_h=\frac{\log 8}{\log 3}=\frac{3\log 2}{\log 3}$ and $d_w\approx 2.097$. However, the H\"older exponent $d_w-d_h$ in \eqref{holder semigroup 2} might not be optimal and it has actually been conjectured in \cite{ABCRST3} that the best H\"older exponent in \eqref{holder semigroup 2} is $d_w-d_h +d_{tH}-1$ where $d_{tH}$ is the topological Hausdorff dimension of the carpet.
 
In this framework, the ingredients \eqref{eq:subGauss-upper2} and \eqref{holder semigroup 2} are enough to repeat the proofs of proposition \ref{estimate G} and theorem   \ref{Holder Riesz}.  The proof of theorem \ref{regularite} also extends to this setting. As a consequence one obtains the following theorem valid under the assumptions of this section.

\begin{theorem}
 Let $W$ be a white noise on $L^2_0(K,\mu)$. Then,  if $\frac{d_h}{2d_w} < s < 1-\frac{d_h}{2d_w}$, there exists a Gaussian random field $(X(x))_{x \in K}$   which is  H\"older continuous  with exponent $H^{-}$ where
\[
H= sd_w -\frac{d_h}{2},
\]
such that for every $f$ which is in the $L^2_0$ domain of the operator $(-\Delta )^s  $
\[
\int_{K} (-\Delta )^s f (x) X(x)  d\mu(x)=\int_{K}  f (x) W (dx).
\]
\end{theorem}

\bibliographystyle{amsplain}

\bibliography{fBm_Refs}

\end{document}